\documentclass[envcountsame]{llncs}
\usepackage{amsmath}
\usepackage{amssymb}
\usepackage{enumitem}

\def\squareforqed{\hbox{\rlap{$\sqcap$}$\sqcup$}}
\def\qed{\ifmmode\squareforqed\else{\unskip\nobreak\hfil
\penalty50\hskip1em\null\nobreak\hfil\squareforqed
\parfillskip=0pt\finalhyphendemerits=0\endgraf}\fi}

\def\lsigma#1{{\Sigma}^0_{#1}}
\def\lpi#1{{\Pi}^0_{#1}}
\def\name#1{{\ulcorner{#1}\urcorner}}

\def\cal#1{{\mathcal #1}}

\def\IN{{\mathbb N}}
\def\IQ{{\mathbb Q}}

\def\baire{\IN^{\IN}}

\def\calP{{\cal P}}

\def\I#1{{{\mathrm{\mathbf I}}(#1)}}

\def\PL{{\mathrm{\mathbf A}}}
\def\PU{{\mathrm{\mathbf K}}}
\def\PO{{\mathrm{\mathbf O}}}

\def\fL{{f_L}}
\def\gL{{g_L}}
\def\fU{{f_U}}
\def\gU{{g_U}}

\def\calPfin{{\mathcal P}_{\mathrm{fin}}}

\def\int{{\mathrm{\mathbf{int}}}}

\def\wayabovearrow{\rlap{\raise-.25ex\hbox{$\shortuparrow$}}\raise.25ex\hbox{$\shortuparrow$}}
\def\waybelowarrow{\rlap{\raise.25ex\hbox{$\shortdownarrow$}}\raise-.25ex\hbox{$\shortdownarrow$}}

\begin{document}

\title{Some notes on spaces of ideals\\ and computable topology}
\author{Matthew de Brecht\thanks{This work was supported by JSPS Core-to-Core Program, A. Advanced Research Networks and by JSPS KAKENHI Grant Number 18K11166. The author thanks Tatsuji Kawai, Takayuki Kihara, Arno Pauly, Matthias Schr\"{o}der, Victor Selivanov, and Hideki Tsuiki for helpful discussions.}}
\institute{Graduate School of Human and Environmental Studies, Kyoto University, Japan\\
\email{matthew@i.h.kyoto-u.ac.jp}}

\maketitle


\section{Introduction}

It was shown in \cite{DPS} that the \emph{quasi-Polish spaces} introduced in \cite{dbr} can be equivalently characterized as spaces of ideals in the following sense. 

\begin{definition}[see \cite{DPS}]\label{def:idealspace}
Let $\prec$ be a transitive relation on $\IN$. A subset $I\subseteq \IN$ is an \emph{ideal} (with respect to $\prec$) if and only if:
\begin{enumerate}
\item
$I \not=\emptyset$,\hfill (\emph{$I$ is non-empty})
\item
$(\forall a \in I) (\forall b \in \IN)\, (b \prec a \Rightarrow b \in I)$,\hfill (\emph{$I$ is a lower set})
\item
$(\forall a,b \in I)(\exists c\in I)\, (a \prec c  \,\&\,  b \prec c)$.\hfill (\emph{$I$ is directed})
\end{enumerate}
The collection $\I{\prec}$ of all ideals has the topology generated by basic open sets of the form $[n]_{\prec} = \{ I \in \I{\prec} \mid n \in I\}$.
\qed
\end{definition}

We often apply the above definition to other countable sets with the implicit assumption that it has been suitably encoded as a subset of $\IN$. If $\prec$ is actually a partial order, then the definition of ideal above agrees with the usual definition of an ideal from order theory. Note that $\I{\prec} \subseteq \bigcup_{n\in\IN}[n]_{\prec}$ and if $I \in [a]_{\prec}\cap[b]_{\prec}$ then there is $c\in\IN$ with $I \in [c]_{\prec} \subseteq [a]_{\prec}\cap[b]_{\prec}$, so $\{ [n]_{\prec} \mid n\in\IN\}$ really is a basis for $\I{\prec}$ and not just a subbasis. Also note that proving the claim in the previous sentence requires all three of the axioms that define ideals.

We first give some basic examples. If $=$ is the equality relation on $\IN$, then $\I{=}$ is homeomorphic to $\IN$ with the discrete topology. If $\prec$ is the strict prefix relation on the set $\IN^{<\IN}$ of finite sequences of natural numbers, then $\I{\prec}$ is homeomorphic to the Baire space $\baire$. If $\subseteq$ is the usual subset relation on the set $\calPfin(\IN)$ of finite subsets of $\IN$, then $\I{\subseteq}$ is homeomorphic to $\calP(\IN)$, the powerset of the natural numbers with the Scott-topology. 

Spaces of the form $\I{\prec}$ for some transitive computably enumerable (c.e.) relation on $\IN$ provide an effective interpretation of quasi-Polish spaces. This effective interpretation as spaces of ideals was first investigated in \cite{DPS}, where they are called \emph{precomputable quasi-Polish spaces}, but they are equivalent to the \emph{computable quasi-Polish spaces} in \cite{KK17}, and they naturally correspond to c.e. propositional geometric theories via the duality in \cite{H15} (see \cite{Chen19} for extending this duality beyond propositional logic). In many applications it is useful to assume that $\I{\prec}$ also comes with a c.e. set $E_\prec = \{ n\in\IN \mid [n]_{\prec}\not=\emptyset \}$, which provides an effective interpretation of \emph{overt quasi-Polish spaces}. These are called \emph{computable quasi-Polish spaces} in \cite{DPS}, and are equivalent to the \emph{effective quasi-Polish spaces} in \cite{HRSS19}, and correspond to \emph{effectively enumerable computable quasi-Polish spaces} in the terminology of \cite{KK17}. Dually, they correspond to c.e. propositional geometric theories where satisfiability is semidecidable.

In this paper, we show some basic results on spaces of ideals, with an emphasis on the connections with computable topology. We also hope that our approach will help clarify the relationship between quasi-Polish spaces and domain theory (see \emph{abstract basis} in \cite{g03} or \cite{GL13}), and implicitly demonstrate how the theory of quasi-Polish spaces can be developed within relatively weak subsystems of second-order arithmetic (see the work on \emph{poset spaces} in \cite{mummert:phd}).


\section{Computable functions}

Computability of functions between spaces of ideals can be defined in a way that is compatible with the TTE framework \cite{W00}. We briefly review the TTE approach to computability on countably based $T_0$-spaces, but see \cite{Sch02} for the extension to the cartesian closed category of admissibly represented spaces and \cite{P16} for more general represented spaces.

Given a countably based $T_0$-space $X$ with fixed basis $(B_i)_{i\in\IN}$, the \emph{standard (admissible) representation} of $X$ is the partial function $\delta_X:\subseteq \baire\to X$ defined as $\delta_X(p) = x \iff range(p) = \{ i\in\IN \mid x \in B_i\}$. A function $f\colon X \to Y$ between spaces with standard admissible representations is computable if and only if there is a  computable (partial) function $F:\subseteq\baire\to\baire$ such that $f\circ \delta_X = \delta_Y\circ F$. It follows that a function $f\colon \I{\prec_1} \to \I{\prec_2}$ is computable if and only if there is an algorithm which transforms any enumeration of the elements of any $I\in \I{\prec_1}$ into an enumeration of the elements of $f(I)\in \I{\prec_2}$. 

We define a code for a partial function to be any subset $R\subseteq \IN\times\IN$. Each code $R$ encodes the partial function $\name{R}:\subseteq \I{\prec_1}\to \I{\prec_2}$ defined as
\begin{eqnarray*}
\name{R}(I) &=& \{ n\in\IN \mid (\exists m\in I)\,\langle m,n\rangle\in R\},\\
dom(\name{R}) &=& \{ I \in \I{\prec_1} \mid \name{R}(I) \in \I{\prec_2}\}.
\end{eqnarray*}

\begin{theorem}
Let $\prec_1$ and $\prec_2$ be transitive relations on $\IN$. A total function $f\colon \I{\prec_1} \to \I{\prec_2}$ is computable if and only if there is a c.e. code $R\subseteq \IN\times\IN$ such that $f=\name{R}$.
\end{theorem}
\begin{proof}
It is clear that if $f=\name{R}$ for some c.e. code $R$, then there is an algorithm which transforms any enumeration of the elements of any $I\in \I{\prec_1}$ into an enumeration of the elements of $f(I)\in \I{\prec_2}$. Therefore, $f$ is computable.

For the other direction, assume $f\colon \I{\prec_1} \to \I{\prec_2}$ is computable. It is a standard result that there is a computable enumeration $(U_n)_{n\in\IN}$ of c.e. subsets of $\IN$ such that $f^{-1}([n]_{\prec_2}) = \bigcup_{m \in U_n} [m]_{\prec_1}$. Let $R = \{\langle m,n\rangle \mid m \in U_n \}$. Given $I\in \I{\prec_1}$, if $n\in\name{R}(I)$, then there is some $m\in I$ with $\langle m,n\rangle\in R$, hence $m \in U_n$. Thus $I \in [m]_{\prec_1} \subseteq f^{-1}([n]_{\prec_2})$ which implies $n\in f(I)$. Conversely, if $n\in f(I)$ then $I \in f^{-1}([n]_{\prec_2})$, so there must be $m\in U_n$ with $I \in [m]_{\prec_1}$. It follows that $\langle m,n\rangle \in R$ and that $n \in \name{R}(I)$. Therefore, $R$ is a c.e. code satisfying $f = \name{R}$.
\qed
\end{proof}


\section{Basic constructions}


\subsection{Products}

Given relations $\prec_1$ and $\prec_2$ on $\IN$, define the relation $\prec^\times_{1,2}$ on $\IN$ as 
\[\langle a,b\rangle \prec^\times_{1,2} \langle a',b'\rangle \iff a \prec_1 a' \text{ \& } b \prec_2 b',\]
where $\langle \cdot, \cdot \rangle \colon \IN\times\IN\to\IN$ is a computable bijection. Then $\I{\prec^\times_{1,2}}$ is computably homeomorphic to the product $\I{\prec_1}\times \I{\prec_2}$ via the pairing function $\langle \cdot,\cdot\rangle \colon \I{\prec_1}\times \I{\prec_2}\to \I{\prec^\times_{1,2}}$
\[\langle I_1, I_2 \rangle = \{\langle a,b\rangle \mid a\in I_1 \text{ \& } b\in I_2\}\]
and the projections $\pi_i \colon \I{\prec^\times_{1,2}} \to \I{\prec_i}$ ($i \in\{1,2\}$)
\begin{eqnarray*}
\pi_1(I) &=& \{ a \in \IN \mid (\exists b\in\IN) \langle a,b\rangle\in I\},\\
\pi_2(I) &=& \{ b \in \IN \mid (\exists a\in\IN) \langle a,b\rangle\in I\}.
\end{eqnarray*}
We leave most of the proof to the reader as an exercise, but we will show that $\pi_1(I)$ really is a lower set because it is a nice example of how directedness and transitivity often compensate for the lack of reflexivity of the relations. Assume $I\in \I{\prec^\times_{1,2}}$ and $a \in \pi_1(I)$ and $a_0 \prec_1 a$. Then there is $b\in\IN$ with $\langle a,b\rangle \in I$. If $\prec_2$ was reflexive, then we would have $\langle a_0,b\rangle \prec^\times_{1,2} \langle a,b\rangle \in I$, and since $I$ is a lower set we would immediately conclude $\langle a_0,b\rangle\in I$. But without reflexivity, we must instead use the directedness of $I$ to first obtain $\langle a',b'\rangle \in I$ with  $\langle a,b\rangle \prec^\times_{1,2} \langle a',b'\rangle$, and then we have $\langle a_0,b\rangle \prec^\times_{1,2} \langle a',b'\rangle \in I$ by the transitivity of $\prec_1$ and $\prec_2$. We still get the desired conclusion $\langle a_0,b\rangle\in I$ (hence $a_0 \in \pi_1(I)$), albeit with a slight detour that required directedness and transitivity.

A simple modification of Definition~2.3.13 in \cite{mummert:phd} can be used to construct countable products from an enumeration $(\prec_i)_{i\in\IN}$ of transitive relations.


\subsection{Co-products}

We get co-products (i.e., disjoint unions) by defining the relation $\prec^+_{1,2}$ on $\IN$ as
\[\langle a,i\rangle \prec^+_{1,2} \langle a',j\rangle \iff  i=j\in\{1,2\} \text{ \& } a \prec_i a'.\]
Then it is an easy exercise to show that $\I{\prec^+_{1,2}}$ is computably homeomorphic to the co-product $\I{\prec_1}+ \I{\prec_2}$. It should be clear to the reader how to extend this to countable co-products.


\subsection{$\lpi 2$-subspaces and equalizers}

Let $\prec$ be a transitive relation on $\IN$. The $\lsigma 1$-subsets (or c.e. open subsets) of $\I{\prec}$ are encoded by c.e. subsets $U\subseteq \IN$ by defining
\[\name{U} = \bigcup_{n\in U}[n]_{\prec}.\]
The $\lpi 2$-subsets of $\I{\prec}$ are encoded by computable enumerations $(U_i, V_i)_{i\in\IN}$ of c.e. subsets of $\IN$ by defining
\[\name{(\forall i)U_i\Rightarrow V_i}= \{ I \in \I{\prec} \mid (\forall i\in\IN)[ I \in \name{U_i} \Rightarrow I \in \name{V_i}]\}.\]

The next theorem is part of the characterization of precomputable quasi-Polish spaces from \cite{DPS}. We provide a direct proof for convenience. 

\begin{theorem}[see \cite{DPS}]\label{thrm:equalizers}
Let $\prec$ be a transitive c.e. relation on $\IN$. Given a code of a $\lpi 2$-subset $A$ of $\I{\prec}$, one can computably obtain a transitive c.e. relation $\sqsubset$ on $\IN$ such that $\I{\sqsubset}$ is computably homeomorphic to $A$.
\end{theorem}
\begin{proof}
Assume $A = \name{(\forall i)U_i\Rightarrow V_i}$ for some computable enumeration $(U_i, V_i)_{i\in\IN}$ of c.e. subsets of $\IN$. Let $\prec^{(\cdot)}$ be a decidable relation such that
\begin{eqnarray*}
m \prec n &\iff& (\exists k\in\IN)\, m \prec^{(k)} n, \text{ and}\\
k\leq k' \text{ \& } m \prec^{(k)} n &\Longrightarrow& m \prec^{(k')} n.
\end{eqnarray*}
Let $(U_i^{(k)})_{i,k\in\IN}$ be a double enumeration of decidable subsets of $\IN$ such that 
\[U_i = \bigcup_{k\in\IN} U_i^{(k)} \text{ and } k\leq k' \Rightarrow U_i^{(k)} \subseteq U_i^{(k')}.\]

For $F_1,F_2 \in \calPfin(\IN)$ and $k_1,k_2\in\IN$, define $\langle F_1, k_1 \rangle \sqsubset \langle F_2, k_2\rangle$ if and only if the following all hold:
\begin{enumerate}
\item
$k_1 < k_2$
\item
$F_1 \subseteq F_2$
\item
$F_2\not=\emptyset$
\item
$(\forall m \leq k_1)\left[[(\exists n\in F_1)\,m\prec^{(k_1)} n] \Rightarrow m\in F_2\right]$
\item
$(\forall a,b\in F_1)(\exists c \in F_2)[ a \prec c \text{ \& } b \prec c]$
\item
$(\forall i \leq k_1)[F_1 \cap U_i^{(k_1)}\not=\emptyset \Rightarrow F_2\cap V_i \not=\emptyset]$.
\end{enumerate} 

It is clear that $\sqsubset$ is c.e., and the monotonicity assumptions on $\prec^{(\cdot)}$ and $U_i^{(k)}$ imply that if $\langle F_1, k_1\rangle \sqsubset \langle F_2,k_2\rangle$ and $F\subseteq F_1$ and $k \leq k_1$ then $\langle F, k\rangle \sqsubset \langle F_2,k_2\rangle$, hence $\sqsubset$ is transitive. 

Define $f\colon \I{\sqsubset} \to \name{(\forall i)U_i\Rightarrow V_i}$ by $f(I) = \bigcup_{\langle F,k\rangle \in I} F$. Given $I\in \I{\sqsubset}$, the directedness of $I$ implies any $\langle F_0, k_0\rangle \in I$ can be extended to a finite $\sqsubset$-chain in $I$ of arbitrary length, hence for any $k\in\IN$ there are $\langle F_1, k_1\rangle,\langle F_2, k_2\rangle \in I$ with $\langle F_0, k_0\rangle \sqsubset\langle F_1, k_1\rangle \sqsubset \langle F_2, k_2\rangle$ and $k < k_1$. Thus $\langle F_0, k \rangle  \sqsubset \langle F_2,k_2\rangle$, which implies $\langle F_0, k \rangle \in I$. It follows that $n\in f(I)$ if and only if $\langle \{n\}, k\rangle \in I$ for some (equivalently, every) $k\in\IN$. By using the directedness of $I$ this observation generalizes from singletons to all finite sets, and so for any $F\in\calPfin(\IN)$ we have $F\subseteq f(I)$ if and only if $\langle F,k\rangle\in I$ for some (equivalently, every) $k\in\IN$. Then conditions 3, 4, and 5 in the definition of $\sqsubset$ imply that $f(I)$ is indeed an ideal of $\prec$, and condition 6 implies $f(I)\in\name{(\forall i)U_i\Rightarrow V_i}$. Thus $f$ is well-defined, and it is clearly a computable injection.

A computable inverse of $f$ is given by $g\colon \name{(\forall i)U_i\Rightarrow V_i} \to \I{\sqsubset}$ defined as $g(I) = \{\langle F,k\rangle \mid k\in\IN \text{ \& $F\subseteq I$ is finite}\}$. The only part of the proof that $g$ is well-defined which requires a little thought is showing that $g(I)$ is directed for each $I\in \name{(\forall i)U_i\Rightarrow V_i}$, but it is not difficult to see that if $\langle F_1,k_1\rangle, \langle F_2, k_2\rangle \in g(I)$, then one can find a finite $G\subseteq I$ which contains $F_1\cup F_2$ and enough of $I$ to satisfy conditions 3 through 6 and obtain $\langle F_1,k_1\rangle, \langle F_2, k_2\rangle \sqsubset \langle G, k_1+k_2+1\rangle \in g(I)$. The claim that $g$ is an inverse to $f$ follows from the observations in the previous paragraph.
\qed
\end{proof}

If $R$ and $S$ are codes for total functions $\name{R},\name{S}\colon \I{\prec_1} \to \I{\prec_2}$, then for any $I \in \I{\prec_1}$ we have $\name{R}(I) = \name{S}(I)$ if and only if 
\[(\forall n\in\IN)\left[ n\in\name{R}(I) \iff n\in\name{S}(I)\right].\]
This is a $\lpi 2$-subset of $\I{\prec_1}$ whenever $\prec_1$, $\prec_2$, $R$, and $S$ are c.e. It follows that we can computably obtain a c.e. relation $\sqsubset$ such that $\I{\sqsubset}$ is an equalizer of $\name{R}$ and $\name{S}$.

Note that Theorem~\ref{thrm:equalizers} is the best result possible because if $\prec$, $\sqsubset$, $R$, and $S$ are c.e. such that $\name{R}\colon \I{\sqsubset} \to \I{\prec}$ is total with partial inverse $\name{S}\colon \I{\prec} \to \I{\sqsubset}$ (meaning $(\forall J\in\I{\sqsubset})[ J = \name{S}(\name{R}(J))]$) then $I\in range(\name{R})$ if and only if $I\in dom(\name{S})$ and $I = \name{R}(\name{S}(I))$. Since $dom(\name{S})$ is $\lpi 2$, it follows that $\I{\sqsubset}$ is computably homeomorphic to the $\lpi 2$-subset $range(\name{R})$ of $\I{\prec}$.


\section{Examples from computable topology}


\subsection{Completion of (computable) separable metric spaces}

Let $(X,d)$ be a separable metric space. Fix a countable dense subset $D\subseteq X$, and define a transitive relation $\prec$ on $P = D\times\IN$ as 
\[\langle x,n\rangle \prec \langle y, m\rangle \iff d(x,y) < 2^{-n}-2^{-m}.\]
This definition guarantees that the open ball with center $x$ and radius $2^{-n}$ contains the closed ball with center $y$ and radius $2^{-m}$. The pair $(P,\prec)$ is a countable substructure of the \emph{formal balls} of $(X,d)$, a well-known construction in domain theory (see Section~V-6 of \cite{g03} and Section~7.3 of \cite{GL13}). It is straightforward to see that $\prec$ is transitive by using the triangular inequality for $d$.

If $I \in \I{\prec}$ then it contains a cofinal infinite ascending $\prec$-chain $(\langle x_i, n_i \rangle)_{i\in\IN}$, which means that $\langle x_i,n_i \rangle\prec \langle x_{i+1}, n_{i+1} \rangle$ for all $i\in\IN$ and that for any $\langle x,n \rangle\in I$ there is $i\in \IN$ with $\langle x,n \rangle \prec \langle x_i,n_i \rangle$. Note that $(n_i)_{i\in\IN}$ is strictly increasing because $0\leq d(x_i, x_{i+1}) < 2^{-n_i} - 2^{-n_{i+1}}$, and therefore $(x_i)_{i\in\IN}$ is a Cauchy sequence. It follows that $\lim_{i\to\infty} d(x,x_i)$ is well-defined for all $x\in D$.

Next we show that $\langle x,n \rangle\in I$ if and only if $\lim_{i\to\infty} d(x,x_i)< 2^{-n}$. For any $\langle x,n \rangle \in I$ the cofinality of $(\langle x_i, n_i \rangle)_{i\in\IN}$ implies there is $i_0\in\IN$ with $\langle x,n \rangle \prec \langle x_i, n_i \rangle$ for all $i\geq i_0$. Let $\varepsilon >0$ be such that $d(x,x_{i_0})  = 2^{-n} - 2^{-n_{i_0}} - \varepsilon$. Then for $i\geq i_0$ we have 
\[d(x, x_i) \leq d(x, x_{i_0}) + d(x_{i_0}, x_i) < (2^{-n} - 2^{-n_{i_0}}-\varepsilon) + (2^{-n_{i_0}} - 2^{-n_i}) < 2^{-n}-\varepsilon,\]
hence $\lim_{i\to\infty} d(x,x_i) \leq 2^{-n} -\varepsilon <  2^{-n}$. Conversely, assume $x\in D$ and there is $\varepsilon >0$ such that $\lim_{i\to\infty} d(x,x_i) < 2^{-n} - \varepsilon$. Fix $i \in \IN$ such that $d(x,x_i)+\varepsilon < 2^{-n}$ and $2^{-n_i} < \varepsilon$. Then $d(x,x_i) < d(x,x_i)+\varepsilon-2^{-n_i} < 2^{-n}-2^{-n_i}$, hence $\langle x,n \rangle \prec \langle x_i,n_i \rangle$, which implies $\langle x,n \rangle\in I$.

It is now easy to see that $\I{\prec}$ is homeomorphic to the completion $(\widehat{X},\widehat{d})$ of $(X,d)$. The usual admissible representation for $\widehat{X}$ is to represent each $x \in\widehat{X}$ by the fast Cauchy sequences $(x_i)_{i\in\IN}$ in $D$ that converge to $x$ (by \emph{fast Cauchy} we mean $d(x_i,x_i+1)<2^{-(i+1)}$ for each $i\in\IN$). From an enumeration of $I\in \I{\prec}$ we can extract a cofinal infinite ascending $\prec$-chain $(\langle x_i, n_i \rangle)_{i\in\IN}$ in $I$ so that $(x_i)_{i\in\IN}$ is a fast Cauchy sequence determining a point in $\widehat{X}$. In the other direction, given a fast Cauchy sequence $(x_i)_{i\in\IN}$ in $D$, we have $d(x_i,x_{i+1}) < 2^{-(i+1)} = 2^{-i} - 2^{-(i+1)}$ for each $i\in\IN$,  hence $(\langle x_i, i \rangle)_{i\in\IN}$ is an infinite ascending $\prec$-chain which generates an ideal $I\in \I{\prec}$. This determines a homeomorphism between $\I{\prec}$ and $\widehat{X}$.

A computable metric space $(X,d)$ comes with an indexing $\alpha\colon \IN \to D$ for some dense $D\subseteq X$ in such a way that $\{ (q,r,i,j) \in \IQ^2\times\IN^2 \mid q < d(\alpha(i),\alpha(j)) < r\}$ is computably enumerable. Defining $\langle i,n \rangle \prec \langle j, m\rangle$ if and only if $d(\alpha(i),\alpha(j)) < 2^{-n}-2^{-m}$ determines a transitive c.e. relation $\prec$ such that $\I{\prec}$ and $\widehat{X}$ are computably homeomorphic.


\subsection{Completion of computable topological spaces}

A (countably based) \emph{computable topological space} (also called an \emph{effective topological space}; see \cite{KK08,Sel08,WG09,KK17,HRSS19}) is a tuple $(X, \varphi, S)$ where: 
\begin{enumerate}
\item
$X$ is a $T_0$-space (we write $\PO(X)$ for its topology),
\item
$\varphi\colon\IN\to\PO(X)$ is an enumeration of a basis for $X$,
\item
$S \subseteq \IN^3$ is a c.e. set satisfying $\varphi(n) \cap \varphi(m) = \bigcup\{ \varphi(k) \mid \langle n,m,k \rangle \in S\}$ for each $n,m\in\IN$.
\end{enumerate}

Note that the only effective aspect of this definition is the c.e. set $S$, and there are no specifications as to how the space $X$ and the enumeration $\varphi$ should be defined. As a result, if $(X,\varphi, S)$ is a computable topological space, then for any subspace $Y \subseteq X$ we can restrict $\varphi$ in the obvious way to obtain a map $\varphi'\colon \IN\to\PO(Y)$ such that $(Y,\varphi', S)$ is also a computable topological space. A common extension of the above definition additionally requires that $\{ n\in\IN\mid \varphi(n)\not=\emptyset\}$ is a c.e. set, but even in this case one can define highly non-constructive dense subspaces of a computable topological space which are still computable topological spaces. 

Since the effective part of the above definition is compatible with infinitely many computable topological spaces, a natural question to ask is whether there is any \emph{canonical} computable topological space associated to a given c.e. set $S$. This question leads to Definition~\ref{def:canonicalcomptop} below. In the following, for any continuous function $f\colon X \to Y$, the function $\PO(f) \colon \PO(Y) \to \PO(X)$ is defined as $\PO(f)(U) = f^{-1}(U)$.

\begin{definition}\label{def:canonicalcomptop}
Let $S\subseteq \IN^3$ be a c.e. set. A computable topological space $(X,\varphi, S)$ is \emph{complete} if and only if for any computable topological space $(Y,\psi,S)$ there is a unique computable embedding $e\colon Y \to X$ satisfying $\psi = \PO(e)\circ \varphi$.
\end{definition}

Intuitively, $(X,\varphi, S)$ is a complete computable topological space if and only if all other computable topological spaces associated to $S$ are essentially just restrictions of the kind $(Y,\varphi', S)$ we saw earlier. Also note that any complete computable topological space associated to $S$ is unique up to computable homeomorphism. The next lemma shows that every c.e. subset $S\subseteq \IN^3$ determines a complete computable topological space.

\begin{lemma}\label{lemma:canonicalcomptop}
For any c.e. subset $S\subseteq \IN^3$, there is a $\lpi 2$-subspace $X \subseteq \calP(\IN)$ such that $\varphi\colon \IN \to \PO(X)$ defined as $\varphi(n) = \{x\in X\mid n\in x\}$ is an enumeration of a basis for $X$ and $(X, \varphi, S)$ is a complete computable topological space.
\end{lemma}
\begin{proof}
Let $S\subseteq \IN^3$ be a c.e. subset. Define $X\subseteq \calP(\IN)$ so that $x\in X$ if and only if the following conditions are all satisfied:
\begin{enumerate}[label=(\roman*)]
\item
$x\not=\emptyset$,
\item
$(\forall \langle n,m,k \rangle \in S)\,[ k\in x \Rightarrow \{n,m\}\subseteq x]$, and
\item
$(\forall n,m\in \IN)\,[ \{n,m\}\subseteq x \Rightarrow (\exists k\in x)\, \langle n,m,k \rangle \in S]$.
\end{enumerate}

It is clear that $X$ is a $\lpi 2$-subspace of $\calP(\IN)$. We first show that $\varphi$ is an enumeration of a basis for $X$. It is clear that each $\varphi(n)$ is an open subset of $X$, and that $\{ \varphi(n)  \mid n\in\IN \}$ covers $X$ because each $x\in X$ is non-empty. Next, note that if $\langle n,m,k \rangle\in S$, then condition (ii) implies $\varphi(k) \subseteq \varphi(n)\cap \varphi(m)$. So for any $x\in \varphi(n)\cap\varphi(m)$, by using condition (iii) it follows that there is $k\in\IN$ with $x\in \varphi(k) \subseteq \varphi(n)\cap\varphi(m)$. Therefore, $\varphi$ is an enumeration of a basis for $X$. It is then easy to see (using condition (iii) again), that $(X,\varphi,S)$ is a computable topological space. 

Given another computable topological space $(Y,\psi,S)$, define $e\colon Y \to X$ as $e(y) = \{ n\in\IN \mid y \in \psi(n)\}$. We first show that $e(y) \in X$ for each $y\in Y$. Clearly, $e(y)$ is non-empty because the basis enumerated by $\psi$ must cover $Y$. Next, if $\langle n,m,k \rangle\in S$ then $\psi(k)$ must be a subset of $\psi(n)\cap \psi(m)$ by condition (3) of the definition of a computable topological space, hence $e(y)$ satisfies condition (ii). Finally, condition (iii) is satisfied because if $\{n,m\}\subseteq e(y)$ then $y \in \psi(n)\cap\psi(m)$ hence there must be $\langle n,m,k \rangle\in S$ with $y\in\psi(k)$ which implies $k\in e(y)$. Therefore, $e$ is well-defined.

Using the fact that $\psi$ enumerates a basis for $Y$, it is easy to see that $e$ is a computable topological embedding. Furthermore, $y \in \psi(n)$ if and only if $n \in e(y)$ if and only if $e(y) \in \varphi(n)$, hence $\psi(n) = e^{-1}(\varphi(n))$. This proves that $\psi = \PO(e)\circ \varphi$, and it is clear that $e$ is the only possible embedding of $Y$ into $X$ that satisfies this property.
\qed
\end{proof}

We take a brief moment to consider the extension where a computable topological space comes with an additional c.e. set $E =\{n\in\IN \mid \varphi(n)\not=\emptyset\}$. Completeness is defined as in Definition~\ref{def:canonicalcomptop}, but with quantification over spaces of the form $(Y,\psi,S,E)$. There is no guarantee that arbitrarily chosen $S$ and $E$ will be compatible, but if they are compatible with at least one computable topological space, then a complete space can be obtained by adding a fourth ($\lpi 2$) axiom ``$(\forall n\in\IN)\,[n\in x \Rightarrow n\in E]$'' to the construction in the proof of Lemma~\ref{lemma:canonicalcomptop}. These modifications could also be made to the following theorem, which shows that complete computable topological spaces provide an effective interpretation of quasi-Polish spaces that is equivalent to the approach using spaces of ideals.

\begin{theorem}\label{thrm:canonicalcomptop}
Every complete computable topological space is computably homeomorphic to $\I{\prec}$ for some transitive c.e. relation $\prec$ on $\IN$. Conversely, given a transitive c.e. relation $\prec$ on $\IN$ one can computably obtain a c.e. subset $S\subseteq \IN^3$ such that $(\I{\prec}, \varphi_{\prec}, S)$ is a complete computable topological space, where $\varphi_{\prec}\colon\IN\to\PO(\I{\prec})$ is the standard enumeration of a basis for $\I{\prec}$ given by $\varphi_{\prec}(n) = [n]_{\prec}$.
\end{theorem}
\begin{proof}
The first claim follows from Lemma~\ref{lemma:canonicalcomptop} and Theorem~\ref{thrm:equalizers}.

For the converse, let $\prec$ be a transitive c.e. relation on $\IN$. Define 
\[S = \{\langle n,m,k \rangle\in\IN^3 \mid n\prec k \text{ and } m\prec k\},\]
and let $(X,\varphi, S)$ be the complete computable topological space for $S$ as constructed in the proof of Lemma~\ref{lemma:canonicalcomptop}. The proof will be completed by showing that $X = \I{\prec}$ as subsets of $\calP(\IN)$.

First we show $\I{\prec} \subseteq X$. Fix $I \in \I{\prec}$. It is clear that $I$ satisfies condition (i) of the definition of $X$. Next, condition (ii) is satisfied because if $\langle n,m,k \rangle \in S$ and $k\in I$, then $n,m \prec k$ by the definition of $S$, hence $\{n,m\}\subseteq I$ because $I$ is a lower set. Finally, condition (iii) is satisfied because if $\{n,m\}\subseteq I$ the directedness of $I$ implies there is $k\in I$ with $n,m\prec k$, hence $\langle n,m,k \rangle\in S$. Therefore, $I\in X$.

To show $X \subseteq \I{\prec}$, fix any $x \in X$. Clearly $x$ is non-empty. Next, assume $k \in x$ and $n \prec k$. Then $\langle n,n,k \rangle \in S$, hence condition (ii) on $X$ implies $n\in x$, so $x$ is a lower set. Finally, if $n,m \in x$ then condition (iii) on $X$ implies there is $\langle n,m,k\rangle \in S$ with $k\in x$. By definition of $S$ we have $n \prec k$ and $m\prec k$, which shows that $x$ is directed. Therefore, $x \in \I{\prec}$.
\qed
\end{proof}

The above theorem shows that we get a computably equivalent definition of computable topological space if we simply define them to be a pair $(\prec, X)$, where $\prec$ is a transitive c.e. relation and $X\subseteq \I{\prec}$. A more rigorous approach would also require a precise definition of the set $X$, for example by defining a (countably based) ``computable topological space'' to be a pair $(\prec, \Phi_X)$ that contains an explicit (finite) formula $\Phi_X$ with a single free variable $I$ that defines the set $X = \{ I \in \I{\prec} \mid \Phi_X(I)\}$ within some fixed formal system. This would lead us more into the realm of effective descriptive set theory, but adopting such a definition would guarantee that \emph{computable} topological spaces are unambiguously defined by a finite amount of information.


\section{Powerspaces}
Given a topological space $X$, we write $\PL(X)$ for the lower powerspace of $X$ (the closed subsets of $X$ with the lower Vietoris topology), and $\PU(X)$ for the upper powerspace of $X$ (the saturated compact subsets of $X$ with the upper Vietoris topology). Our notation follows that of \cite{DK19}, where other basic results on quasi-Polish powerspaces can be found. For countably based spaces, the lower powerspace defined here is equivalent to the space of (closed) overt sets in \cite{P16,DPS}. In this section, we show how to represent powerspaces as spaces of ideals using the construction introduced in \cite{sm83} for $\omega$-algebraic domains (which is equivalent to the case that $\prec$ is a partial order within our framework). We fix a transitive relation $\prec$ on $\IN$ for the rest of this section. 

\subsection{Lower powerspace}

A basis for the lower Vietoris topology on $\PL(\I{\prec})$ is given by sets of the form 
\[\bigcap_{n\in F}\Diamond [n]_{\prec} = \{ A \in \PL(\I{\prec}) \mid (\forall n\in F)(\exists I \in A)\, n \in I\}\]
for $F \in \calPfin(\IN)$. Define the transitive relation $\prec_L$ on $\calPfin(\IN)$ as
\[F \prec_L G \text{ if and only if } (\forall m\in F)\,(\exists n\in G)\, m \prec n.\]
Transitivity of $\prec_L$ easily follows from the transitivity of $\prec$, and it is clear that $\prec_L$ is c.e. whenever $\prec$ is. Next, define $\fL \colon \PL(\I{\prec}) \to \I{\prec_L}$ as
\[\fL(A) = \{ F \in \calPfin(\IN) \mid (\forall m\in F)(\exists I\in A)\, m\in I\}\]
and $\gL \colon \I{\prec_L} \to \PL(\I{\prec})$ as 
\[\gL(J) = \{ I \in \I{\prec} \mid (\forall m\in I)(\exists F \in J)\, m \in F\}.\]
We will need the following lemma when we prove that these two functions are well-defined computable homeomorphisms.

\begin{lemma}\label{lem:lowerpowerspacelemma}
If $J \in \I{\prec_L}$ and $F \in\calPfin(\IN)$,  then $(\forall m\in F)\, \gL(J) \cap [m]_{\prec}\not=\emptyset$ if and only if $F \in J$.
\end{lemma}
\begin{proof}
First assume $(\forall m\in F)\, \gL(J) \cap [m]_{\prec}\not=\emptyset$. For each $m \in F$, there is $I \in \gL(J)$ with $m \in I$, and as $I$ is directed, there is $n \in I$ with $m \prec n$, but since $I \in \gL(J)$ there must be $G \in J$ with $n \in G$, and therefore $\{m\} \prec_L G \in J$. This shows that $\{ m\} \in J$ for each $m \in F$. Since $F$ is finite and $J$ is directed, there is $H \in J$ such that $(\forall m\in F)\, \{m\} \prec_L H$. It follows that $F \prec_L H$, and therefore $F \in J$.

For the converse, assume $F \in J$, and fix any $m \in F$. Since $J$ is directed there exists an infinite sequence $F = F_0 \prec_L F_1 \prec_L F_2 \prec_L \cdots$ with $F_i \in J$ for each $i\in \IN$. From the definition of $\prec_L$, there exists an infinite sequence $m = m_0 \prec m_1 \prec m_2 \prec \cdots$ with $m_i \in F_i$ for each $i\in \IN$. Then $I = \{ n \in \IN \mid (\exists i\in\IN)\, n \prec m_i \}$ is in $\I{\prec}$ and $m\in I$. For any $n\in I$ there is $i\in \IN$ with $n \prec m_i \in F_i$, thus $\{n\} \prec_L F_i \in J$ which implies $\{n\} \in J$. Therefore, $I \in \gL(J) \cap [m]_{\prec}$.
\qed
\end{proof}

\begin{theorem}
$\PL(\I{\prec})$ and $\I{\prec_L}$ are computably homeomorphic.
\end{theorem}
\begin{proof}
We will prove that $\fL$ and $\gL$ are well-defined computable inverses of each other in several steps.
\begin{itemize}[label=\textbullet]
\item
\emph{$\fL$ is well-defined:} We must show that $\fL(A)$ is an ideal.
\begin{enumerate}
\item
\emph{($\fL(A)$ is non-empty).} $\fL(A) \not=\emptyset$ because $\emptyset \in \fL(A)$. 
\item
\emph{($\fL(A)$ is a lower set).} If $G \in \fL(A)$ and $F \prec_L G$, then for any $m\in F$ there is $n \in G$ with $m \prec n$. There is some $I \in A$ with $n\in I$, and also $m \in I$ because $I$ is a lower set. Therefore, $F \in \fL(A)$. 
\item
\emph{($\fL(A)$ is directed).} Assume $F, G \in \fL(A)$. For each $m \in F \cup G$ there is some $I \in A$ with $m \in I$, and by directedness of $I$ we can choose some $n_m \in I$ with $m \prec n_m$. Combine these choices into a single (finite) set $H = \{ n_m \mid m \in F\cup G\}$. Then $H \in \fL(A)$ and $F,G \prec_L H$.
\end{enumerate}
\item
\emph{$\gL$ is well-defined:} We must show that $\gL(J)$ is a closed subset of $\I{\prec}$. If $I \not \in \gL(J)$, then by definition of $\gL(J)$ there must be $m \in I$ such that $(\forall F\in J)\, m\not\in F$. Then $[m]_{\prec}$ is an open neighborhood of $I$ that does not intersect $\gL(J)$, hence $\gL(J)$ is closed.
\item
\emph{$\fL$ is computable:} Clearly, $\fL(A) \in [F]_{\prec_L}$ if and only if $A \in {\bigcap}_{m\in F} \Diamond [m]_{\prec}$.
\item
\emph{$\gL$ is computable:} Lemma~\ref{lem:lowerpowerspacelemma} is the statement $\gL(J) \in \bigcap_{m\in F} \Diamond [m]_{\prec}$ if and only if $J \in [F]_{\prec_L}$. 
\item
\emph{$\fL(\gL(J)) = J$:} The above proofs that $\fL$ and $\gL$ are computable imply that $F \in \fL(\gL(J))$ if and only if $\gL(J) \in {\bigcap}_{m\in F} \Diamond [m]_{\prec}$ if and only if $F \in J$.
\item
\emph{$\gL(\fL(A)) = A$:} The above proofs that $\gL$ and $\fL$ are computable imply that $\gL(\fL(A)) \in \bigcap_{m\in F} \Diamond [m]_{\prec}$ if and only if $F \in \fL(A)$ if and only if $A \in \bigcap_{m\in F} \Diamond [m]_{\prec}$.
\end{itemize}
\qed
\end{proof}

\subsection{Upper powerspace}

A basis for the upper Vietoris topology on $\PU(\I{\prec})$ is given by sets of the form \[\Box \bigcup_{n\in F} [n]_{\prec} = \{ K \in \PU(\I{\prec}) \mid (\forall I \in K)(\exists n \in F)\, n \in I\}\]
for $F \in \calPfin(\IN)$. Define the transitive relation $\prec_U$ on $\calPfin(\IN)$ as
\[F \prec_U G \text{ if and only if }(\forall n\in G)\,(\exists m\in F)\, m \prec n.\]
Transitivity of $\prec_U$ easily follows from the transitivity of $\prec$, and it is clear that $\prec_U$ is c.e. whenever $\prec$ is. Next, define $\fU \colon \PU(\I{\prec}) \to \I{\prec_U}$ as
\[\fU(K) = \{ F \in \calPfin(\IN) \mid (\forall I \in K)(\exists m\in F)\, m\in I\}\]
and $\gU \colon \I{\prec_U} \to \PU(\I{\prec})$ as 
\[\gU(J) = \{ I \in \I{\prec} \mid (\forall F\in J)(\exists m\in I)\, m \in F \}.\]
We will need the following lemma when we prove that these two functions are well-defined computable homeomorphisms.

\begin{lemma}\label{lem:upperpowerspacelemma}
If $J \in \I{\prec_U}$ and $S\subseteq \IN$,  then $\gU(J) \subseteq \bigcup_{m\in S} [m]_{\prec}$ if and only if there is finite $F\subseteq S$ with $F \in J$.
\end{lemma}
\begin{proof}
For the easy direction, assume $F\subseteq S$ is finite and $F \in J$. Then every $I \in \gU(J)$ intersects $F$, which implies $\gU(J) \subseteq \bigcup_{m\in F} [m]_{\prec} \subseteq \bigcup_{m\in S} [m]_{\prec}$.

Conversely, assume $\gU(J) \subseteq \bigcup_{m\in S} [m]_{\prec}$. Since $J$ is an ideal and countable, there is a sequence $(F_i)_{i\in\IN}$ in $J$ satisfying $(\forall i\in\IN)\, F_i \prec_U F_{i+1}$ and $(\forall F\in J)(\exists i\in\IN)\, F \prec_U F_i$. It is straightforward to see that $I \in \gU(J)$ if and only if $(\forall i\in\IN)\,F_i\cap I \not=\emptyset$. Define 
$T$ to be the set of all $\sigma \in\IN^{<\IN}$ satisfying:
\begin{enumerate}
\item
$(\forall i<len(\sigma)-1)\,\sigma(i) \prec \sigma(i+1)$,
\item
$(\forall i<len(\sigma))\,\sigma(i)\in F_i$,
\item
$(\forall i<len(\sigma))(\forall m\in S)\,m\not\prec \sigma(i)$.
\end{enumerate}
Clearly $T$ is closed under subsequences, hence $T$ is a finitely branching tree because of item 2. If $T$ contained an infinite path $p$ then the ideal $I = \{ n \in\IN \mid (\exists i\in\IN)\, n \prec p(i) \}$ would be in $\gU(J)$ even though item $3$ prevents $I$ from being in $\bigcup_{m\in S} [m]_{\prec}$, which would be a contradiction. It follows from K\"{o}nig's lemma that $T$ is finite. Let $k\in\IN$ be an upper bound for $\{ len(\sigma) \mid \sigma\in T\}$. 

Assume for a contradiction that there is $n_k\in F_k$ such that $(\forall m\in S)\, m \not\prec n_k$. If $k>0$, then $F_{k-1} \prec_U F_k$, hence there is $n_{k-1} \in F_{k-1}$ with $n_{k-1} \prec n_k$, and transitivity of $\prec$ implies $(\forall m\in S)\, m \not\prec n_{k-1}$. Continuing in this way, we can construct a finite sequence $\sigma \in \IN^{<\IN}$ as $\sigma(k) = n_k$, $\sigma(k-1)=n_{k-1}$, and so on, in such a way that $\sigma \in T$ but $len(\sigma) = k+1$, which contradicts the choice of $k$. 

Therefore, for each $n \in F_k$ there is $m_n \in S$ with $m_n \prec n$. Then $F = \{ m_n \mid n \in F_k\}$ is a finite subset of $S$ satisfying $F \prec_U F_k$, hence $F \in J$.
\qed
\end{proof}

\begin{theorem}
$\PU(\I{\prec})$ and $\I{\prec_U}$ are computably homeomorphic.
\end{theorem}
\begin{proof}
We will prove that $\fU$ and $\gU$ are well-defined computable inverses of each other in several steps.
\begin{itemize}[label=\textbullet]
\item
\emph{$\fU$ is well-defined:} We must show that $\fU(K)$ is an ideal.
\begin{enumerate}
\item
\emph{($\fU(K)$ is non-empty).} Ideals are non-empty, so we can fix some $m_I \in I$ for each $I \in K$. By compactness of $K$ there is a finite subset $F$ of $\{ m_I \mid I \in K\}$ such that $K \subseteq \bigcup_{m_I \in F} [m_I]_{\prec}$. Then $F \in \fU(K)$, hence $\fU(K) \not=\emptyset$.
\item
\emph{($\fU(K)$ is a lower set).} Assume $G \in \fU(K)$ and $F \prec_U G$. For any $I\in K$ there exists $n \in G \cap I$, and since $F \prec_U G$ there is $m \in F$ with $m \prec n$. Then $m \in F \cap I$ because $I$ is a lower set, and it follows that $F \in \fU(K)$.
\item
\emph{($\fU(K)$ is directed).} Assume $F, G \in \fU(A)$. For each $I \in K$ there exist $m_I \in F \cap I$ and $n_I \in G \cap I$. Since $I$ is an ideal, there is $p_I \in I$ with $m_I \prec p_I$ and $n_I \prec p_I$. By compactness of $K$ there is a finite subset $H$ of $\{ p_I \mid I \in K\}$ such that $K \subseteq \bigcup_{p_I \in H} [p_I]_{\prec}$. Then $H \in \fU(K)$ and $F, G \prec_U H$.
\end{enumerate}
\item
\emph{$\gU(J)$ is well-defined:} We must show that $\gU(J)$ is a saturated compact subset of $\I{\prec}$. It is clear that $\gU(J)$ is saturated, because the specialization order on $\I{\prec}$ is subset inclusion, and if $I$ intersects each $F \in J$ then so does any superset $I'$ of $I$. To show compactness, assume $S \subseteq \IN$ is such that $\gU(J) \subseteq \bigcup_{m\in S} [m]_{\prec}$. Using Lemma~\ref{lem:upperpowerspacelemma}, there is finite $F\subseteq S$ with $F \in J$, hence $\gU(J) \subseteq \bigcup_{m\in F} [m]_{\prec}$.
\item
\emph{$\fU$ is computable:} Clearly, $\fU(K) \in [F]_{\prec_U}$ if and only if $K \in \Box \bigcup_{m\in F} [m]_{\prec}$.
\item
\emph{$\gU$ is computable:} Lemma~\ref{lem:upperpowerspacelemma} implies $\gU(J) \in \Box \bigcup_{m\in F} [m]_{\prec}$ if and only if $J \in [F]_{\prec_U}$. 
\item
\emph{$\fU(\gU(J)) = J$:} The above proofs that $\fU$ and $\gU$ are computable imply that $F\in \fU(\gU(J))$ if and only if $\gU(J) \in \Box \bigcup_{m\in F} [m]_{\prec}$ if and only if $F \in J$.
\item
\emph{$\gU(\fU(K)) = K$:} The above proofs that $\gU$ and $\fU$ are computable imply that $\gU(\fU(K)) \in \Box \bigcup_{m\in F} [m]_{\prec}$ if and only if $F \in \fU(K)$ if and only if $K \in \Box \bigcup_{m\in F} [m]_{\prec}$.
\end{itemize}
\qed
\end{proof}


\bibliographystyle{amsplain}
\bibliography{myrefs}

\end{document}